\documentclass [a4paper, 12pt]{amsart}

\usepackage[margin=2.5cm]{geometry} 
\usepackage{amsmath}
\usepackage{amsfonts}
\usepackage{amssymb}
\usepackage{amsthm}
\usepackage{graphicx}
\usepackage{pb-diagram}
\usepackage{epstopdf}
\usepackage{fancyhdr}
\usepackage[all]{xy}

\newtheorem{cor}{Corollary}[section]
\newtheorem{te}[cor]{Theorem}
\newtheorem{p}[cor]{Proposition}

\newtheorem{lemma}[cor]{Lemma}
\theoremstyle{definition}
\newtheorem{de}[cor]{Definition}
\theoremstyle{remark}
\newtheorem{ob}[cor]{Observation}

\newtheorem{nt}[cor]{Notation}

\newcommand{\cz}{\mathbb{C}}
\newcommand{\nz}{\mathbb{N}}
\newcommand{\zz}{\mathbb{Z}}
\newcommand{\qz}{\mathbb{Q}}

\newcommand{\ff}{\mathbb{F}}
\newcommand{\ee}{\mathcal{E}}
\newcommand{\bb}{\mathcal{B}}

\newcommand{\unit}{\mathcal{U}}

\newcommand{\vp}{\varepsilon}

   {\begin{verse}%
     \small }%
   {\end{verse}}

\begin{document}

\title{Convex Structures Revisited}
\maketitle
\begin{center}
Liviu P\u aunescu\footnote{Work supported by a grant of the Romanian National Authority for Scientific Research, CNCS - UEFISCDI, project number PN-II-ID-PCE-2012-4-0201}
\end{center}

 $\mathbf{Abstract.}$ We provide a complete characterisation of extreme points of the space of sofic representations. We also show that the restriction map $Sof(G,P^\omega)$ to $Sof(H,P^\omega)$, where $H\subset G$ is not always surjective. The first part of the paper is a continuation of \cite{Pa2} and follows more closely the plan of Nathanial  Brown from \cite{Br}. 

\tableofcontents

\section{Introduction}

The goal of this article is to review results in \cite{Pa2} and to obtain a complete characterisation of extreme points, similar to the one in \cite{Br}. There were two obstacles preventing us from obtaining a complete characterisation. 

Firstly, amplifications were a big issue. We managed to prove a link between extreme points and ergodicity of the commutant largely due to Proposition 2.8. It proved that ergodicity is preserved under taking amplifications. In order to get an "if and only if" statement, one would need a converse to Proposition 2.8. Sadly this kind of questions proved to be rather hard to settle. Also, in \cite{Br}, there was no need to ask these questions, due to the diffuse nature of the hyperfinite factor, as opposed to our finite dimensional object $P_n$ (permutation matrices). This problem will be solved by considering "type $II_1$" permutations, that is the full group of an amenable type $II_1$ equivalence relation.

Secondly, it was hard to construct elements in the commutant of a sofic representation. Much to my surprise, there are sofic representations that act ergodically on the Loeb space (the sofic representation itself is ergodic, not its commutant). Such sofic representations are necessary extreme points and it seems that there are no tools to construct enough elements in the commutant to get ergodicity. This problem can only be solved by restricting the Loeb space to the commutant of the sofic representation.

Note however that the result in \cite{Pa2}, Theorem 2.10 is still useful in that form. Ergodicity of the communtant of a sofic representation is a question that appears in the study of sofic entropy, see \cite{Ke-Li}.

Throughout the article $\omega$ denotes, as usual, a free ultrafilter on $\nz$. We work with $M_n(\cz)$, the algebra of matrices in dimension $n$ and its special subsets $P_n$, the subgroup of permutation matrices and $D_n$, the subalgebra of diagonal matrices. This time we also need $R$, the unique hyperfinite type $II_1$ factor. We assume familiarity with ultraproducts of finite von Neumann algebras, there are many places in literature where the construction can be checked.

\subsection{The convex strucutre on $Hom(N,R^\omega)$}

Lets recall the construction from \cite{Br}. Let $N$ be a separable type $II_1$ factor. The following theorem is a fruitful result in the theory of type $II_1$ factors:

\begin{te}\label{mcduff-jung}
The factor $N$ is the hyperfinite factor if and only if any two unital homomorphism  $\pi,\rho:N\to R^\omega$ are unitary conjugate, i.e. 
there exists  $u\in\unit(R^\omega)$ such that $\pi(a)=u\rho(a)u^*$ for all $a\in N$.
\end{te}

The direct implication is classic, while the converse is a much recent result due to Jung, \cite{Ju}. The question is now what happens outside of the hyperfinite world?
One can always consider the set:
\[Hom(N,R^\omega)=\{\pi:N\to R^\omega:\mbox{unital homeomorphism}\}/\sim,\]
where $\sim$ is unitary conjugacy defined as in the statement of Theorem \ref{mcduff-jung}.

This space has a natural topology given by pointwise convergence in the weak topology of $R^\omega$. Due to the separability of $N$ this turns out to be a metrizable topology. 
Ozawa showed in the appendix of \cite{Br} that for non-hyperfinite $N$ the space $Hom(N,R^\omega)$ is non-separable (given that $N$ satisfies Connes' Embedding Conjecture, 
i.e. $Hom(N,R^\omega)$ is non-empty). This is quite an unpleasant fact, being just another example where the hyperfinite case is completly separated from the rest of the world. 

Still, if you want to study the set $Hom(N,R^\omega)$, the first observation is that the direct sum operation is constructing new elements out of old ones. Let  $\pi,\rho:N\to R^\omega$
and consider the direct sum:
\[\pi\oplus\rho:N\to R^\omega\oplus R^\omega=(R\oplus R)^\omega.\]
Choose a unital embedding $\theta:R\oplus R\to R$ so that $\theta^\omega:(R\oplus R)^\omega\to R^\omega$. Then:
\[\theta^\omega\circ(\pi\oplus\rho):N\to R^\omega\]
is a unital embedding whose class in $Hom(N,R^\omega)$ can be different than the classes of $\pi$ and $\rho$.

The map $\theta:R\oplus R\to R$ is unital so $\theta(1\oplus 1)=1$. Notice that $\theta(1\oplus 0)$ is a projection in $R$. Let $\lambda=Tr(\theta(1\oplus 0))$. Then $Tr(\theta(0\oplus 1))$ must equal $1-\lambda$. Denote by $[\xi]$ the class in $Hom(N,R^\omega)$ of a unital morphism $\xi:N\to R^\omega$. We set by definition:
\[[\theta^\omega\circ(\pi\oplus\rho)]=\lambda[\pi]+(1-\lambda)[\rho].\]
The term $\lambda[\pi]+(1-\lambda)[\rho]$ is just a formal writing for the element $[\theta^\omega\circ(\pi\oplus\rho)]$ that we constructed. Of course there are some 
well-defined problems to be solved, but nothing more than routine work. The last observation to be made is that one can construct a map $\theta$ for any prescribed $\lambda\in [0,1]$.

It can be checked that if $[\pi]\neq[\rho]$ and $\lambda\neq 1$ then  $\lambda[\pi]+(1-\lambda)[\rho]\neq[\pi]$ as expected. Furhermore there are some axioms
involving also a metric for the topology on $Hom(N,R^\omega)$ that have to be settled. After this axioms are checked one can deduce due to a result by Capraro and Fritz (\cite{Ca-Fr}) that $Hom(N,R^\omega)$
together with its metric and convex structures can be regarded as an honest closed convex subset of a Banach space.

Now we can state the main result of Brown's theory.

\begin{te}
The class $[\pi]\in Hom(N,R^\omega)$ is an extreme point if and only if the relative comutant $N'\cap R^\omega$ is a factor.
\end{te}

This is a nice result describing the extreme points of this convex structure, but their existence is still an open problem.

\subsection{The space $Sof(G,P^\omega)$}

In \cite{Pa2}, we replaced the separable factor $N$ and Connes' Embedding Conjecture by a countable group $G$ and by the sofic property respectively. We review here the construction of $Sof(G,P^\omega)$ from the article but, though this paper is pretty much self-contained, I'm assuming some familiarity with notations and results from \cite{Pa2}. 

We want to study embeddings of the group $G$ into the universal sofic group $\Pi_{k\to\omega}P_{n_k}$ that are "trace preserving", i.e. the trace of each nontrivial element is $0$. We call such morphisms \emph{sofic representations} of $G$. We first note that we have a similar result to Theorem \ref{mcduff-jung} due to Elek and Szabo, \cite{El-Sz2}:

\begin{te}\label{elek-szabo}
The group $G$ is amenable if and only if for any two group morphisms $\Theta_1,\Theta_2:G\to \Pi_{k\to\omega}P_{n_k}$ such that $Tr(\Theta_i(g))=0$ for any $g\neq e$ and any $i=1,2$, 
there exists $p\in\Pi_{k\to\omega}P_{n_k}$ such that $\Theta_2(g)=p\Theta_1(g)p^*$ for any $g\in G$.
\end{te}

In order to be able to construct a convex structure on the set of sofic representations, we have to be flexible, considering universal sofic groups over any sequence of dimensions $\{n_k\}_k$ such that $n_k\to\infty$. This brings in certain complications, as we want to compare sofic representations over different universal sofic groups $\Pi_{k\to\omega}P_{n_k}$ and $\Pi_{k\to\omega}P_{m_k}$.

We notice that for a sequence $\{r_k\}_k\subset\nz^*$ the universal sofic group $\Pi_{k\to\omega}P_{n_k}$ canonically embeds in $\Pi_{k\to\omega}P_{n_kr_k}$ by tensoring by identity:
\[\Pi_{k\to\omega}P_{n_k}\ni\Pi_{k\to\omega}p_k\to\Pi_{k\to\omega}(p_k\otimes 1_{r_k})\in\Pi_{k\to\omega}P_{n_kr_k}.\]

Let $\Theta:G\to\Pi_{k\to\omega}P_{n_k}$ be a sofic approximation, $\Theta=\Pi_{k\to\omega}\theta_k$. We call an \emph{amplification} of $\Theta$ the composition of $\Theta$ with the above canonic map:
\[\Theta\otimes 1_{r_k}:G\to\Pi_{k\to\omega}P_{n_kr_k}\ \ \Theta\otimes 1_{r_k}(g)=\Pi_{k\to\omega}\theta_k(g)\otimes 1_{r_k}.\]

The space of sofic representations is now:
\[Sof(G,P^\omega)=\{\Theta:G\to\Pi_{k\to\omega}P_{n_k}:\mbox{sofic representation } (n_k)_k\subset\nz,n_k\to\infty\}/\sim,\]
where $\sim$ is amplifications and conjugacy as in Theorem \ref{elek-szabo}. So two sofic representations are equivalent if they have amplifications that are conjugate. For $\Theta:G\to\Pi_{k\to\omega}P_{n_k}$ we denote by $[\Theta]_P$ its class in $Sof(G,P^\omega)$, to distinguish it from $[\Theta]_\ee$ that we shall construct.

\section{Type $II_1$ permutations}

\subsection{The $I_n$ case} One way of discussing sofic objects is by starting with a probability measure preserving equivalence relation. Usually one takes the full equivalence relation on a space with $n$ elements, endowed with the normalised cardinal measure. The Feldman-Moore construction of this equivalence relation is the Cartain pair $D_n\subset M_n(\cz)$, where $D_n$ is the subalgebra of diagonal matrices. The full group of this type $I_n$ equivalence relation is $Sym(n)$, the symmetric group. It embeds in $M_n$ as $P_n$, the subgroup of permutation matrices. If $x\in Sym(n)$ then the corresponding element in $P_n$ is $\tilde x(i,j)=\delta_i^{p(j)}$ or $\tilde x$ is the characteristic function of the graph of $p^{-1}$.

The group $P_n$ is in the normaliser of $D_n$, i.e. if $p\in P_n$ and $a\in D_n$ then $pap^*\in D_n$.  What we have here is just a symmetric group, $P_n$, acting on a set with $n$ elements in an obvious way. Passing to ultraproducts things become more interesting:
\[\Pi_{k\to\omega}P_{n_k}\curvearrowright\Pi_{k\to\omega}D_{n_k}.\]
The group $\Pi_{k\to\omega}P_{n_k}$ was introduced by Elek and Sazbo (\cite{El-Sz1}) and it is called \emph{the universal sofic group}. A countable group is \emph{sofic} if and only if it is a subgroup of this group.

The algebra $\Pi_{k\to\omega}D_{n_k}$ is an abelian von Neumann algebra, isomorphic to $L^\infty(X_\omega,\mu_\omega)$, where $(X_\omega,\mu_\omega)$ is a Loeb measure space, i.e. an ultra product of probability spaces (for its construction see \cite{Lo} or \cite{El-Sze}).

I call the action itself $\Pi_{k\to\omega}P_{n_k}\curvearrowright\Pi_{k\to\omega}D_{n_k}$ \emph{the universal sofic action}. By definition, a standard action is \emph{sofic} if it can be embedded into a universal sofic action.

\subsection{The hyperfinite case} Instead of the space $\{1,\ldots,n\}$ with the normalised cardinal measure, we consider the unit interval $[0,1]$ endowed with the Lebesgue measure $\mu$. On this space consider the equivalence relation:
\[E=\{(x,y):x-y\in\qz\}\]
It is a standard fact that $E$ is a measurable, countable, $\mu$-preserving, ergodic, amenable equivalence relation. A consequence of the famous Connes-Feldman-Weiss theorem (\cite{CFW}) is the unicity of such an object: there is a unique ergodic, amenable type $II_1$ equivalence relation corresponding to the unique hyperfinite type $II_1$ factor.

We denote by $[E]$ \emph{the full group} of $E$:
\[[E]=\{u:[0,1]\to[0,1]:u\mbox{ bijection }, (x,u(x))\in E\mbox{ for $\mu$-almost every $x$}\}.\]
We still have a Hamming distance on $[E]$ defined by $d_H(u,v)=\mu(\{x:u(x)\neq v(x)\})$. The Feldman-Moore construction of $E$, by definition, consists of some functions from $E$ to $\cz$ (we don't go here into details, they are not so important):
\[M(E)=\{f:E\to\cz: f\mbox{ is a multiplier}\}.\] 
Operations on this algebra are defined as follows:
\begin{align*}
f\cdot g(x,z)=&\sum_{yEx}f(x,y)g(y,z);\\
f^*(x,y)=&\overline{f(y,x)}\\
Tr(f)=&\int_Xf(x,x)d\mu(x).
\end{align*}

The Cartain subalgebra of $M(E)$ is composed of those functions with the support on the diagonal:
\[A=\{f:E\to\cz: f\in M(E)\mbox{ and }f(x,y)=0 \mbox{ if } x\neq y\}.\]

It is a standard fact that $A$ is isomorphic to $L^\infty(X,\mu)$ and $M(E)$ is isomorphic to $R$, the hyperfinite type $II_1$ factor.

The full group $[E]$ can be embedded in $M(E)$ and we denote its image by $\ee$ (in the type $I_n$ the full group was $Sym(n)$ and it's image in the Feldman-Moore construction was $P_n$). The group $\ee$ is composed of those functions $f:E\to\cz$ that have exactly one entry of $1$ on each row and column:
\[\ee=\{f\in M(E):f(E)=\{0,1\}\mbox{ and }\forall x\exists! y f(x,y)=1\mbox{ and }\forall y\exists! x f(x,y)=1\}.\]

If $u\in [E]$ denote by $\tilde u=\chi_{graph(u^{-1})}\in\ee$. The formula $d_H(u,v)=1-Tr(\tilde u\tilde v^*)$ can be easily checked. Now we have $\ee$ acting on $A$ inside $M(E)$, to replace the old picture of $P_n$ action on $D_n$ inside $M_n$. Passing to ultraproducts we get the diffuse universal sofic picture:
\[\ee^\omega\curvearrowright A^\omega.\]
This is still a universal sofic action and $\ee^\omega$ is a universal sofic group. The benefit of this objects is that we no longer need amplifications in order to compare two sofic representations.

\subsection{The limit of symmetric groups}

As we said, the group $\ee$ is the type $II_1$ analogue of $P_n$ and the next result will strengthen this idea.

Let $n,r$ be two natural numbers. Then $f_{n,r}:M_n\to M_{nr}$ defined by $f_{n,r}(x)=x\otimes 1_r$ is a trace preserving embedding. More over $f_{n,r}(P_n)\subset P_{nr}$ and $f_{n,r}(D_n)\subset D_{nr}$.

As $d_H(p,q)=1-Tr(pq^*)$, the map $f_{n,r}$ restricted to $P_n$ preserves the Hamming distance. We construct the direct limit of metric groups $\varinjlim (P_n,d_H)$ by taking the metric closure in the Hamming distance of the algebraic limit of the directed system $(P_n,f_{n,r})$. The same construction is available for $(D_n,Tr)$.

\begin{te}\label{directlimit}
The direct limit $\varinjlim (P_n,d_H)$ is isomorphic to $[E]$. 
\end{te}
\begin{proof}
For $n\in\nz$ embed $P_n$ into $[E]$ by dividing the interval $[0,1]$ in $n$ equal parts and permuting these small intervals. The formula will look something like:
\[\Phi_n:P_n\to[E],\ \ \Phi_n(p)=x\to\frac{p([nx])+\{nx\}}{n},\]
where $[y],\{y\}$ are the integer and fractional part of $y$ and permutations $p\in P_n$ act on the set $\{0,1,\ldots,n-1\}$. It is easy to check that the maps $\Phi_n$ preserve the metrics and they are compatible with the directed system. It follows that we have an embedding $\Phi:\varinjlim P_n\to[E]$. To show that this map is surjective we need to check that $\bigcup_n\Phi_n(P_n)$ is dense in $[E]$.

For this, let $\phi\in[E]$ and $\vp>0$. For $q\in\qz\cap[0,1)$ let
\[A_q=\{x\in[0,1]:\phi(x)=x+q\ (mod\ 1)\}.\]
Then $\{A_q:q\in\qz\cap[0,1)\}$ is a partition of $[0,1]$ and we can choose $\{q_1,\ldots,q_k\}$ a finite set such that $\sum_{i=1}^k\mu(A_{q_i})>1-\vp$. To simplify notation, denote $A_{q_i}$ by $A_i$.

Let $\bb_n$ be the $\sigma-algebra$ generated by the sets $[j/n,(j+1)/n)$, where $j=0,\ldots,(n-1)$, such that elements in $\Phi_n(P_n)$ are measurable as functions from $([0,1],\bb_n)$ to $([0,1],\bb_n)$.

Using the regularity properties of the Lebesgue measure, we can find a suficiently large $n$ and sets $B_i'\in\bb_n$ such that $\mu(A_i\Delta B_i')<\vp/k^2$. As sets $\{A_i\}$ are disjoint it follows that $\mu(B_i'\cap B_j')<2\vp/k^2$ for $i\neq j$. Take out overlaping intervals to get disjoint sets $B_i\in\bb_n$ such that $\mu(A_i\Delta B_i)<2\vp/k$.

Increasing $n$ we can assume that each $q_i$ is an integer multiple of $1/n$. Then there exists an element $\psi\in\Phi_n(P_n)$ such that $\psi(x)=x+q_i$ if $x\in B_i$. Then $\psi=\phi$ on $C=\bigcup_{i=1}^k(A_i\cap B_i)$. But $\mu(C)>\sum_{i=1}^k(\mu(A_i)-\vp/k)>1-2\vp$.

It follows that the distance between $\phi$ and $\psi$ is smaller than $2\vp$ and we are done.
\end{proof}

Actually much more is true. By dividing the interval $[0,1]$ in $n$ equal parts we can construct a trace preserving embedding $\Psi_n:M_n(\cz)\to M(E)$. Then $\Psi_n(D_n)\subset A$ and $\Psi_n(P_n)\subset\ee$, the later being actually just $\Phi_n$ from the proof above composed with the canonical isomorphism between $[E]$ and $\ee$. The maps $(\Psi_n)_n$ are also constructing isomorphisms $\varinjlim D_n\simeq A$ and $\varinjlim M_n\simeq M(E)$.

In this article we construct sofic representations in $\ee^\omega$, but, due to the last theorem, most of the time we still deal just with $\Pi_{k\to\omega}P_{n_k}$. This is good for some definitions and it is also intuitive. So the next notation and theorem are quite important for the technical part of the article.

\begin{nt}
For a fixed sequence $\{n_k\}$ construct $\Psi:\Pi_{k\to\omega}M_{n_k}\to M(E)^\omega$ defined by $\Psi(\Pi_{k\to\omega}x_k=\Pi_{k\to\omega}\Psi_{n_k}(x_k)$. 
\end{nt}

By construction $\Psi(\Pi_{k\to\omega}D_{n_k})\subset A^\omega$ and $\Psi(\Pi_{k\to\omega}P_{n_k})\subset\ee^\omega$. Also note that $\Psi$ is trace preserving, in particular it is injective for a fixed sequence $\{n_k\}$.

\begin{te}\label{dense permutation}
Let $\{u_i\}_{i\in\nz}$ be a countable set of elements in $\ee^\omega$ and $\{a_i\}_{i\in\nz}$ be a countable set of elements in $A^\omega$. Then there exists a sequence $\{n_k\}_k$ and elements $v_i\in\Pi_{k\to\omega}P_{n_k}$ and $b_i\in\Pi_{k\to\omega}D_{n_k}$ such that $\Psi(v_i)=u_i$ and $\Psi(b_i)=a_i$ for any $i\in\nz$.
\end{te}
\begin{proof}
The proof is just a consequence of Theorem \ref{directlimit} and of the analogue result $\varinjlim D_n\simeq A$, by a diagonal argument. For simplicity in writing (not in the argument) we only consider the family $\{u_i\}_i$. Let $u_i=\Pi_{k\to\omega}u_i^k$, where $u_i^k\in\ee$. Choose $\{\vp_k\}_k$ strictly positive numbers such that $\vp_k\to_k 0$. 

By Theorem \ref{directlimit}, $\bigcup_n\Psi_n(P_n)$ is dense in $\ee$. Then, for each $k$, there exists $n_k\in\nz$ and $v_1^k,\ldots, v_k^k\in P_{n_k}$ such that $d_H(u_i^k,\Psi_{n_k}(v_i^k))<\vp_k$, for any $i\leqslant k$.
Define $v_i=\Pi_{k\to\omega}v_i^k$. It's clear by construction that $\Psi(v_i)=u_i$.
\end{proof}

\section{The space of diffuse sofic representations}

\begin{de}
A \emph{diffuse sofic representation} is a group morphism $\Theta:G\to\ee^\omega$ such that $Tr(\Theta(g))=0$ for any $g\neq e$.
\end{de}

\begin{de}
For a countable group $G$ denote by $Sof(G,\ee^\omega)$ the space of diffuse sofic representations factored by conjugacy: $\Theta_1\sim \Theta_2$ iff there exists $u\in\ee^\omega$ such that $\Theta_2=Ad u\circ\Theta_1$.
\end{de}

\begin{nt}
For a diffuse sofic representation $\Theta$ we denote by $[\Theta]_\ee$ its class in $Sof(G,\ee^\omega)$.
\end{nt}
 
We prove now that there is a bijection between $Sof(G,P^\omega)$ and $Sof(G,\ee^\omega)$.

\begin{p}
Let $\Theta_1,\Theta_2$ be two sofic representations such that $[\Theta_1]_P=[\Theta_2]_P$. Then $[\Psi\circ\Theta_1]_\ee=[\Psi\circ\Theta_2]_\ee$.
\end{p}
\begin{proof}
Let $\Theta:G\to\Pi_{k\to\omega}P_{n_k}$ be a sofic representation and $\{r_k\}_k\subset\nz^*$. Inspecting the definition of $\Phi_n$ from the proof of Theorem \ref{directlimit} we see that $\Phi_n(s)=\Phi_n(s\otimes 1_r)$, for any $s\in P_n$ and $r\in\nz^*$. It follows that $\Psi(\Theta)=\Psi(\Theta\otimes 1_{r_k})$.

Assume now that $\Theta_1$ and $\Theta_2$ are conjugate. So there exists $u\in\Pi_{k\to\omega}P_{n_k}$ such that $u\Theta_1u^*=\Theta_2$. Then $\Psi(u)\Psi(\Theta_1)\Psi(u)^*=\Psi(\Theta_2)$, implying that $[\Psi(\Theta_1)]_\ee=[\Psi(\Theta_2)]_\ee$.
\end{proof}

\begin{te}
The map $A:Sof(G,P^\omega)\to Sof(G,\ee^\omega)$ defined by $A([\Theta]_P)=[\Psi\circ\Theta]_\ee$ is a bijection.
\end{te}
\begin{proof}
If $\Theta$ is a diffuse sofic representation then  use Theorem \ref{dense permutation} to construct a sofic representation $\Gamma$ with $\Psi(\Gamma)=\Theta$. This shows that $A$ is surjective.

Let now $\Theta_1$ and $\Theta_2$ be sofic representation such that $A([\Theta_1]_P)=A([\Theta_2]_P)$. Then there is $u\in\ee^\omega$ so that $u\Psi(\Theta_1)u^*=\Psi(\Theta_2)$. Again by Theorem \ref{dense permutation} there is $v$ in some $\Pi_{k\to\omega}P_{n_k}$ with $\Psi(v)=u$. Now, amplifying $\Theta_1,\Theta_2$ and $v$ to a common sequence of dimensions and using the injectivity of $\Psi$ we get $(v\otimes 1)(\Theta_1\otimes 1)(v\otimes 1)^*=\Theta_2\otimes 1$. It follows that $[\Theta_1]_P=[\Theta_2]_P$.
\end{proof}

A convex-like structure is defined on a metric space. We transport, via the bijection $A$, the metric defined in \cite{Pa2} Section 1.4.

\begin{de}
Let $G=\{g_0,g_1,\ldots\}$ where $g_0=e$. For $[\Theta_1]_\ee,[\Theta_2]_\ee\in Sof(G,\ee^\omega)$ define:
\[d([\Theta_1],[\Theta_2])=inf\{\big(\sum_{i=1}^\infty\frac1{4^i}||\Theta_1(g_i)-u\Theta_2(g_i)u^*||_2^2\big)^{1/2}:u\in\ee^\omega\}.\]
\end{de}

It follows that $Sof(G,P^\omega)$ and $Sof(G,\ee^\omega)$ are isomorphic as metric spaces, via the map $A$.

\subsection{The direct sum of the universal sofic group}

Let $\lambda\in[0,1]$. We construct a morphism $\Phi_\lambda:\ee^\omega\oplus\ee^\omega\to\ee^\omega$ to be used in the definition of the convex structure. 

 Let $u,v\in\ee^\omega$. Use Theorem \ref{dense permutation} to get $u_1,v_1\in\Pi_{k\to\omega}P_{n_k}$ so that $\Psi(u_1)=u$ and $\Psi(v_1)=v$. Choose two sequences of natural numbers $\{r_k\}_k$ and $\{s_k\}_k$ such that $\lim_{k\to\omega}r_k/(r_k+s_k)=\lambda$. Construct $(u_1\otimes 1_{r_k} )\oplus(v_1\otimes 1_{s_k}):G\to\Pi_{k\to\omega}P_{(r_k+s_k)n_k}$. Define:
 \[\Phi_\lambda(u\oplus v)=\Psi[(u_1\otimes 1_{r_k})\oplus(v_1\otimes 1_{s_k})].\]

Note that $\Phi_\lambda(u\oplus v)$ does not depend on the particular choice of the sequences $\{r_k\}_k$ and $\{s_k\}_k$ as long as $\lim_{k\to\omega}r_k/(r_k+s_k)=\lambda$. The equality $\Psi(x)=\Psi(x\otimes 1)$ for any $x\in M_n(\cz)$ is important here. Also, the ultraproduct construction is factoring out small dependencies.

If $a_\lambda=\chi_{[0,\lambda]}\in A$ (characteristic function) then $\Phi_\lambda(u\oplus v)$ commutes with $(a_\lambda)^\omega\in A^\omega$ for any $u,v\in\ee^\omega$ . This is usual geometry in type $II_1$ factors. It is also a central observation for these convex structures that has to be made a theorem.

\begin{te}\label{image}
The image of $\Phi_\lambda$ is composed of those elements that commutes with $(a_\lambda)^\omega$, where $a_\lambda$ is the characteristic function of $[0,\lambda]$:
\[\Phi_\lambda(\ee^\omega\oplus\ee^\omega)=\{u\in\ee^\omega:ua_\lambda=a_\lambda u\}.\]
\end{te}

\begin{ob}
The definition of $\Phi_\lambda$ can be extended to $M(E)^\omega\oplus M(E)^\omega$. A nice application is the formula $(a_\lambda)^\omega=\Phi_\lambda(1\oplus 0)$. 
\end{ob}

\subsection{Cutting representations}
 
 Cutting a diffuse sofic representation by a projection in $A^\omega$ represents the inverse operation of the direct sum. By Theorem \ref{image}, one needs a projection commuting with $\Theta$ (as $a_\lambda$ is in the role of the projection cutting a corner of the sofic representation).
 
\begin{nt}
Denote by $T_1:\ee^\omega\oplus\ee^\omega\to\ee^\omega$ the projection on the first summand, i.e. $T_1(u\oplus v)=u$. Similarly $T_2(u\oplus v)=v$.
\end{nt}
 
\begin{de}
Let $p$ be a projection in $A^\omega$ commuting with $\Theta$. Let $\lambda=Tr(p)$. Choose an element  $u\in\ee^\omega$ such that $upu^*=a_\lambda$. Define $\Theta_p^u=T_1(\Phi_\lambda^{-1}(u\Theta u^*))$.
\end{de} 
 
\begin{p}
The class of $\Theta_p^u$ does not depend on the choice of $u$.
\end{p} 
 \begin{proof}
 Let $u,v\in\ee^\omega$ be so that $upu^*=a_\lambda=vpv^*$. Then $uv^*$ commutes with $a_\lambda$. We have:
 \begin{align*}
 \Theta_p^u=&T_1(\Phi_\lambda^{-1}(u\Theta u^*))=T_1(\Phi_\lambda^{-1}(uv^*v\Theta v^*vu^*))=T_1[\Phi_\lambda^{-1}(uv^*)\Phi_\lambda^{-1}(v\Theta v^*)\Phi_\lambda^{-1}(vu^*)]\\
 =&T_1(\Phi_\lambda^{-1}(uv^*))\cdot\Theta_p^v\cdot T_1(\Phi_\lambda^{-1}(uv^*))^*.
 \end{align*}
 As $T_1(\Phi_\lambda^{-1}(uv^*))$ is an element of $\ee^\omega$, it follows that $[\Theta_p^u]_\ee=[\Theta_p^v]_\ee$.
 \end{proof}
 
 \begin{de}
 For a projection $p\in A^\omega$ commuting with $\Theta$ define $[\Theta_p]_\ee$ to be the class of $\Theta_p^u$ for a $u\in\ee^\omega$ so that $upu^*=a_{Tr(p)}$.
 \end{de}
 
The following results are useful, both for the proof of the main result and also as an exercise to get the intuition of direct sums and amplifications of diffuse sofic representations.
 
 \begin{lemma}
 Let $p$ be a projection in $\Theta'\cap A^\omega$ with $Tr(p)=\lambda$. Choose an element  $u\in\ee^\omega$ such that $upu^*=a_\lambda$. Then $[T_2(\Phi_\lambda^{-1}(u\Theta u^*))]_\ee=[\Theta_{1-p}]_\ee$.
 \end{lemma}
\begin{proof}
The main observation is that there exists $v\in\ee^\omega$ such that $va_{1-\lambda}v^*=1-a_\lambda$ and $v\Phi_{1-\lambda}(x\oplus y)v^*=\Phi_\lambda(y\oplus x)$ for any $x,y\in\ee^\omega$.

Let $\Theta_1,\Theta_2:G\to\ee^\omega$ be such that $\Theta_1\oplus\Theta_2=\Phi_\lambda^{-1}(u\Theta u^*)$. Then $u\Theta u^*=\Phi_\lambda(\Theta_1\oplus\Theta_2)=v\Phi_{1-\lambda}(\Theta_2\oplus\Theta_1)v^*$. It follows that $\Theta_2=T_1(\Phi_{1-\lambda}^{-1}(v^*u\Theta u^*v))$. As $v^*u(1-p)u^*v=v^*(1-a_\lambda)v=a_{1-\lambda}$, by definition we have $\Theta_2=\Theta_{1-p}^{v^*u}$. As $\Theta_2=T_2(\Phi_\lambda^{-1}(u\Theta u^*))$ it follows that $[T_2(\Phi_\lambda^{-1}(u\Theta u^*))]_\ee=[\Theta_{1-p}]_\ee$.
\end{proof}

In a way the following Proposition is an anti-amplification. This feature is unique to diffuse sofic representations.

\begin{p}\label{anti amplification}
For any diffuse sofic representation $\Theta$ and any $\lambda\in(0,1)$ there exists a projection $p\in A^\omega$, commuting with $\Theta$, such that $Tr(p)=\lambda$ and $[\Theta]_\ee=[\Theta_p]_\ee$
\end{p} 
 \begin{proof}
 Let $\Gamma:G\to\Pi_{k\to\omega}P_{n_k}$ be a sofic representation such that $\Theta=\Psi\circ\Gamma$. Let $\{r_k\}$ be a strictly increasing sequence of natural numbers. Let also $q\in\Pi_{k\to\omega}D_{r_k}$ be a projection such that $Tr(q)=\lambda$. Then $1_{n_k}\otimes q$ commutes with $\Gamma\otimes 1_{r_k}$. Moreover $\Gamma_{1_{n_k}\otimes q}$ is still an amplification of $\Gamma$, so $[\Gamma]_P=[\Gamma_{1_{n_k}\otimes q}]_P$.
 
 Let $p=\Psi(1_{n_k}\otimes q)$. Then $p$ commutes with $\Theta=\Psi(\Gamma)$ and $Tr(p)=\lambda$. Also the equality $[\Gamma]_P=[\Gamma_{1_{n_k}\otimes q}]_P$, transported via $\Psi$, becomes $[\Theta]_\ee=[\Theta_p]_\ee$.
 \end{proof}
 
 One problem in proving that the old action $\alpha(\Theta)$ is ergodic is constructing elements in the commutant $\Theta'$. The last proposition solved this problem, by an easy amplification. In the following lemma we note that there are plenty of projections inside $\Theta'\cap A^\omega$.
 
 \begin{lemma}\label{diffuse abelian}
 The algebra $\Theta'\cap A^\omega$ is diffuse, i.e. it has no minimal projection.
 \end{lemma}
 \begin{proof}
 Let $p\in\Theta'\cap A^\omega$ be a projection. Choose a sequence $\{n_k\}_k$ so that there exists $\Gamma:G\to\Pi_{k\to\omega}P_{n_k}$ a sofic representation and $q\in\Pi_{k\to\omega}D_{n_k}$ a projection such that $\Theta=\Psi\circ\Gamma$ and $p=\Psi(q)$. Because $\Psi$ is injective on $\Pi_{k\to\omega}M_{n_k}$, $q$ commutes with $\Gamma$.
 
 Let $a\in D_2$ be a projection of trace $1/2$. Construct $q\otimes a\in\Pi_{k\to\omega}D_{2n_k}$. This is a projection with $Tr(q\otimes a)=\frac12Tr(q)$ that commutes with $\Gamma\otimes 1_2$. Then $\Psi(q\otimes a)$ commutes with $\Psi(\Gamma\otimes 1_2)=\Theta$ and $\Psi(q\otimes a)$ is a sub-projection of $p$.
 \end{proof}

 \section{The convex structure}
 
 \begin{de}
 For $\Theta_1,\Theta_2$ diffuse sofic representations and $\lambda\in[0,1]$ define:
 \[\lambda[\Theta_1]_\ee+(1-\lambda)[\Theta_2]_\ee=[\Phi_\lambda(\Theta_1\oplus\Theta_2)]_\ee.\]
 \end{de}
 
 At this stage we can consider $\lambda[\Theta_1]_\ee+(1-\lambda)[\Theta_2]_\ee$ to be just a formal notation for the element in $Sof(G,\ee^\omega)$ that we constructed. After the axioms of convex-like structures are checked, we can use Capraro-Fritz theorem to deduce that $Sof(G,\ee^\omega)$ endowed with the metric and this convex structure is a bounded closed convex subset of a Banach space. Then $\lambda[\Theta_1]_\ee+(1-\lambda)[\Theta_2]_\ee$ is a convex combination in  this Banach space.
 
 As an observation, this definition is just the old convex structure on $Sof(G,P^\omega)$ transported on $Sof(G,\ee^\omega)$ via the map $A$. This is enough to deduce that the axions of convex-like structures (see Section 2 of \cite{Br}) are satisfied by $[Sof(G,\ee^\omega),d]$. However, it is easy to check them directly from the definitions presented in this paper. 
 
\begin{p}\label{isolating summand}
If $[\Theta]_\ee=\lambda[\Theta_1]_\ee+(1-\lambda)[\Theta_2]_\ee$ then there exists $p\in A^\omega$ a projection commuting with $\Theta$, $Tr(p)=\lambda$ such that $[\Theta_p]_\ee=[\Theta_1]_\ee$.
\end{p} 
 \begin{proof}
 We can assume that $\Theta=\Phi_\lambda(\Theta_1\oplus\Theta_2)$. Then $p=a_\lambda$. By definition $\Theta_p^{Id}=T_1(\Phi_\lambda^{-1}(\Theta))=\Theta_1$. It follows that $[\Theta_p]_\ee=[\Theta_1]_\ee$.
 \end{proof}

 \begin{p}\label{P3.3.4}(Analog of Proposition 3.3.4 of \cite{Br})
Let $p,q\in \Theta'\cap A^\omega$ be such that $Tr(p)=Tr(q)$. Then $[\Theta_p]=[\Theta_q]$ if and only if there exists an element $u\in\Theta'\cap\ee^\omega$ such that $upu^*=q$.
\end{p} 
\begin{proof}
Let $\lambda=Tr(p)=Tr(q)$ and let $v_p,v_q\in\ee^\omega$ be such that $v_ppv_p^*=a_\lambda=v_qqv_q^*$.
Let $u\in\Theta'\cap\ee^\omega$ be such that $upu^*=q$. Then:
\[(v_quv_p^*)a_\lambda(v_quv_p^*)^*=(v_qu)p(v_qu)^*=v_qqv_q^*=a_\lambda,\]
so $(v_quv_p^*)$ commutes with $a_\lambda$. Let $u_1=T_1(\Phi_\lambda^{-1}(v_quv_p^*))$. Then:
\begin{align*}
u_1\Theta_p^{v_p}u_1^*=&T_1(\Phi_\lambda^{-1}(v_quv_p^*))T_1(\Phi_\lambda^{-1}(v_p\Theta v_p^*))T_1(\Phi_\lambda^{-1}(v_quv_p^*))^*\\
=&T_1[\Phi_\lambda^{-1}(v_quv_p^*v_p\Theta v_p^*(v_quv_p^*)^*)]=T_1[\Phi_\lambda^{-1}(v_q\Theta v_q^*)]=\Theta_q^{v_q}.
\end{align*}
It follows that $[\Theta_p]=[\Theta_q]$.

Assume now that $[\Theta_p]=[\Theta_q]$. By the axioms of the convex-like structures (metric compatibility, see also Proof of Corollary 6 from \cite{Ca-Fr}) it follows that also  $[\Theta_{1-p}]=[\Theta_{1-q}]$. Recall that $[\Theta_p]=[T_1(\Phi_\lambda^{-1}(v_p\Theta v_p^*))]$ and $[\Theta_{1-p}]=[T_2(\Phi_\lambda^{-1}(v_p\Theta v_p^*))]$. So there exists $u_1,u_2\in\ee^\omega$ such that:
\[u_1T_1(\Phi_\lambda^{-1}(v_q\Theta v_q^*))u_1^*=T_1(\Phi_\lambda^{-1}(v_p\Theta v_p^*))\mbox{ and }u_2T_2(\Phi_\lambda^{-1}(v_q\Theta v_q^*))u_2^*=T_2(\Phi_\lambda^{-1}(v_p\Theta v_p^*))\]
Let $u=\Phi_\lambda(u_1\oplus u_2)$. Then $u_1=T_1(\Phi_\lambda^{-1}(u))$ and $u_1=T_2(\Phi_\lambda^{-1}(u))$. We have:
\begin{align*}
v_p\Theta v_p^*=&\Phi_\lambda(\Phi_\lambda^{-1}(v_p\Theta v_p^*))=\Phi_\lambda[T_1(\Phi_\lambda^{-1}(v_p\Theta v_p^*))\oplus T_2(\Phi_\lambda^{-1}(v_p\Theta v_p^*))]\\
=&\Phi_\lambda[u_1T_1(\Phi_\lambda^{-1}(v_q\Theta v_q^*))u_1^*\oplus u_2T_2(\Phi_\lambda^{-1}(v_q\Theta v_q^*))u_2^*]\\
=&\Phi_\lambda[T_1(\Phi_\lambda^{-1}(uv_q\Theta v_q^*u^*))\oplus T_2(\Phi_\lambda^{-1}(uv_q\Theta v_q^*u^*))]\\
=&\Phi_\lambda[\Phi_\lambda^{-1}(uv_q\Theta v_q^*u^*)]=uv_q\Theta v_q^*u^*
\end{align*}
We proved that $v_q^*u^*v_p$ commutes with $\Theta$. As $u$ is in the image of $\Phi_\lambda$ it commutes with $a_\lambda$. It follows that:
\[(v_q^*u^*v_p)p(v_q^*u^*v_p)^*=(v_q^*u^*)a_\lambda(v_q^*u^*)^*=v_q^*a_\lambda v_q=q.\]
\end{proof}

\subsection{Actions on the Loeb space}

In $Sof(G,\ee^\omega)$ there is no need for amplifications. Another difference is that we consider only those elements of the Loeb space that commute with the diffuse sofic representation.

\begin{nt}
For a diffuse sofic representation $\Theta:G\to\ee^\omega$ denote by $\gamma(\Theta)$ the action of $\Theta'\cap\ee^\omega$ on $\Theta'\cap A^\omega$, defined by $\gamma(u)(a)=uau^*$.
\end{nt} 
 
 The following lemma is easy, but it is one of the few tools that allows us to construct permutations. This is why it is so important. It was used in \cite{Pa1} and \cite{Pa2} (Lemma 1.6 in both articles, by a strange coincidence that I am noticing now). Here we need the diffuse version of this lemma. The proof is still the same, using Theorem \ref{dense permutation}.

 \begin{lemma}\label{permutations}
Let $\{p_i:i\in\nz\}$ be projections in $A^\omega$ such that $\sum_ip_i=1$. Let $\{u_i:i\in\nz\}$ be unitary elements in $\ee^\omega$ such that $\sum_iu_i^*p_iu_i=1$. Then
$v=\sum_iu_ip_i$ is an element in $\ee^\omega$.
\end{lemma}

\begin{p}\label{full group}
Let $\Theta$ be a diffuse sofic representation such that $\gamma(\Theta)$ is ergodic. Then if $p,q$ are projections in $\Theta'\cap A^\omega$ such that $Tr(p)=Tr(q)$ then there exists $u\in\Theta'\cap\ee^\omega$ such that $q=upu^*$.
\end{p} 
 \begin{proof}
 Assume first that $pq=0$ (the underlying sets, on which $p$ and $q$ are projecting, are disjoint). We want to construct a partial isometry $v$ such that $vpv^*=q$. As $\gamma(\Theta)$ is ergodic there exists $u\in\Theta'\cap\ee^\omega$ such that $upu^*\cdot q\neq 0$. By a maximality argument we can construct projections $\{p_i\}_i$ and $\{q_i\}_i$ in $\Theta'\cap A^\omega$ and unitaries $\{u_i\}_i$ in $\Theta'\cap\ee^\omega$ such that $\sum_ip_i=p$, $\sum_iq_i=q$ and $u_ip_iu_i^*=q_i$ for any $i$.
 
 Define $v=\sum_iu_ip_i$. It is easy to check that $vpv^*=q$, $vv^*=p$ and $v^*v=q$. Then $u=(1-p-q)+v+v^*$ is a unitary commuting with $\Theta$ such that $upu^*=q$. The proof is algebraic, but there is a lot of geometry behind the scene. The unitary $u$ is sending the underlying set of $p$ onto the underlying set of $q$ and vice-versa, while being identity on the rest of the space.
 
 In order to prove that $u\in\ee^\omega$, use the previous lemma with $\{p_i\}_i\cup\{q_i\}_i\cup\{1-p-q\}$ as the set of projections and $\{u_i\}_i\cup\{u_i^*\}_i\cup\{Id\}$ as the set of unitaries.
 
 If $pq\neq 0$, replace $p$ and $q$ by $p_1=p-pq$ and $q_1=q-pq$.
 \end{proof}

\subsection{The main result}

\begin{p}
Let $\Theta:G\to\ee^\omega$ be a sofic representation. Then $[\Theta]$ is an extreme point in $Sof(G,\ee^\omega)$ if and only if $[\Theta]=[\Theta_p]$ for any projection $p\in \Theta(G)'\cap A^\omega$.
\end{p}
\begin{proof}
This is just a consequence of Proposition \ref{isolating summand}.
\end{proof}

\begin{te}(Analog of Proposition 5.2 of \cite{Br})
Let $\Theta:G\to\ee^\omega$ be a sofic representation. Then $[\Theta]$ is an extreme point in $Sof(G,\ee^\omega)$ if and only if the action $\gamma(\Theta)$ is ergodic.
\end{te}
\begin{proof}
 Assume that $[\Theta]_\ee=\lambda[\Theta_1]_\ee+(1-\lambda)[\Theta_2]_\ee$. Then by Proposition \ref{isolating summand} there exists $p\in\Theta'\cap A^\omega$ a projection with $Tr(p)=\lambda$ such that $[\Theta_p]_\ee=[\Theta_1]_\ee$. Also by Proposition \ref{anti amplification} there exists $q\in\Theta'\cap A^\omega$ a projection with $Tr(q)=\lambda$ such that $[\Theta_q]_\ee=[\Theta]_\ee$. If $\gamma(\Theta)$ is ergodic then by Proposition \ref{full group} there exists $u\in\Theta'\cap\ee^\omega$ such that $upu^*=q$. Use now Proposition \ref{P3.3.4} to deduce that $[\Theta_p]_\ee=[\Theta_q]_\ee$. It follows that $[\Theta]_\ee=[\Theta_1]_\ee$ proving that $[\Theta]_\ee$ is an extreme point.
 
For the converse, let $p,q\in\Theta'\cap A^\omega$ be two projections such that $Tr(p)=Tr(q)$. By the previous proposition $[\Theta_p]_\ee=[\Theta_q]_\ee$. Then, by Proposition \ref{P3.3.4} there exists $u\in\Theta'\cap\ee^\omega$ such that $q=upu^*$. This is enough to deduce the ergodicity of $\gamma(\Theta)$ as $\Theta'\cap A^\omega$ is diffuse (Lemma \ref{diffuse abelian}).
\end{proof}

\begin{ob}
The convex-like structures $Sof(G,P^\omega)$ and $Sof(G,\ee^\omega)$ are isomorphic. This means that the extreme points constructed in Section 2.6 of \cite{Pa2} are still valid for $Sof(G,\ee^\omega)$. The existence of extreme points for any sofic group remains however an open question.
\end{ob}

\section{Sofic representations that cannot be extended}

Let $G=\zz*\zz_2=<a,b:b^2=e>$ and $c=bab$. Then $\ff_2=<a,c>$ is a copy of the free group inside $G$. Let $R:Sof(G,P^\omega)\to Sof(\ff_2,P^\omega)$ be the restriction map $R([\Theta])=[\Theta|_{\ff_2}]$. In this section we show that $R$ is not surjective.

It is quite easy to show that most of the sofic representations $\Theta:\ff_2\to\Pi_{k\to\omega}P_{n_k}$ can not be extended to a sofic representation $\tilde\Theta:G\to\Pi_{k\to\omega}P_{n_k}$ (same sequence of dimensions). A sofic representation of $\ff_2$ is obtained by choosing to sequences of $n_k$-cycles. A sofic representation of $G$ is obtained by choosing to sequences of $n_k$-cycles that are (almost) conjugated by an element of order $2$. As a relatively low number of pairs of cycles are conjugated by an element of order two, it follows that most $\Theta:\ff_2\to\Pi_{k\to\omega}P_{n_k}$ cannot be extended to $G$.

However, when studying the function $R$, we must take into consideration amplifications. Indeed, there are sofic representations $\Theta$ that are not extendable, but they have amplifications that are. For example, assume that there exist an element $y\in\Pi_{k\to\omega}P_{n_k}$ such that $y\Theta(a)y^{-1}=\Theta(c)$ and $y^2\Theta(a)=\Theta(a)y^2$. Then, one can check that, if $\tilde y=
  \left[ {\begin{array}{cc}
   0 & y \\
   y^{-1} & 0 \\
  \end{array} } \right]$, then $\tilde y^2=Id$ and $\tilde y(\Theta(a)\otimes 1_2)\tilde y^{-1}=\Theta(c)\otimes 1_2$.
  
I'm quite sure that the existence of such an element $y\in\Pi_{k\to\omega}P_{n_k}$ is equivalent to the fact that $\Theta\otimes 1_2$ is extendable to $G$. We don't need this result. We shall prove that when $\Theta$ is an expander, which is know to happen most of the time, this is the only phenomena that may make an amplification of $\Theta$ expandable.

\subsection{Hamming distance on matrices}

\begin{de}
For $x,y\in M_n(\cz)$ define the \emph{Hamming distance on matrices}:
\[d_H(x,y)=\frac1n|\{i:\exists j\ x(i,j)\neq y(i,j)\}|.\]
\end{de}
The formula counts the number of rows that are different in $x,y$. Note that if $x,y\in P_n$ then this distance is the usual Hamming distance on a symmetric group.

\begin{de}
We call a matrix $q\in M_n$ a \emph{pice of permutation} if $q$ has only $0$ and $1$ entries and no more than one entry of $1$ on each row and each column. Alternatively $q=pa$, where $p\in P_n$ and $a$ is a projection in $D_n$.
\end{de}

\begin{p}
Let $x,y\in M_n$ and $p\in P_n$. Then $d_H(x,y)=d_H(px,py)=d_H(xp,yp)$. Instead, if $p$ is a pice of permutation then:
\[d_H(px,py)\leqslant d_H(x,y)\mbox{ and }d_H(xp,yp)\leqslant d_H(x,y).\]
\end{p}

The following lemma is the key of the proof. From the existence of an element $y\in P_{nr}$ with some properties, we interfere the existence of an element $w\in P_n$ with similar properties. These type of results we are looking for.

\begin{lemma}\label{main lemma}
Let $x,z\in P_n$ and $y\in P_{nr}$ be such that $y^2=Id_{nr}$ and $d_H(y(x\otimes 1_r),(z\otimes 1_r)y)<\vp$. Assume that for any projection $p\in D_n$, $Tr(p)<1/2$ implies $\lambda Tr(p)<d_H(p,xpx^*)+d_H(p,zpz^*)$. Then there exists $w\in P_n$ such that $d_H(wx,zw)<72\vp/\lambda$ and $d_H(xw,wz)<72\vp/\lambda$. 
\end{lemma}
\begin{proof}
As $M_{nr}\simeq M_r\otimes M_n$, elements in $M_{nr}$ can be viewed as functions from $\{1,\ldots r\}^2$ to $M_n$. Then $(x\otimes 1_r)(i,j)=\delta_i^jx$ and $[y(x\otimes 1_r)](i,j)=\sum_ky(i,k)(x\otimes 1_r)(k,j)=y(i,j)x$. Similarly $[(z\otimes 1_r)y](i,j)=z\cdot y(i,j)$.

Let $A,B\in P_{nr}$. We want to compare $d_H(A,B)$ to $\sum_{i,j=1}^rd_H(A(i,j),B(i,j))$. If $A$ and $B$ are different on a row it may happen that we count twice this error in $\sum_{i,j=1}^rd_H(A(i,j),B(i,j))$. It follows that:
\[2d_H(A,B)\geqslant\frac1r\sum_{i,j=1}^rd_H(A(i,j),B(i,j)).\]
By hypothesis $d_H(y(x\otimes 1_r),(z\otimes 1_r)y)<\vp$ and we can also deduce $d_H((x\otimes 1_r)y,y(z\otimes 1_r))<\vp$. Let $d_H(y(i,j)x,zy(i,j))=\vp_{i,j}^1$ and $d_H(xy(i,j),y(i,j)z)=\vp_{i,j}^2$. Then:
\begin{align*}
\frac1r\sum_{i,j=1}^r\vp_{i,j}^1&\leqslant 2d_H(y(x\otimes 1_r),(z\otimes 1_r)y)<2\vp;\\
\frac1r\sum_{i,j=1}^r\vp_{i,j}^2&\leqslant 2d_H((x\otimes 1_r)y,y(z\otimes 1_r))<2\vp.\\
\end{align*}
From these inequalities we can deduce the existence of an $i\in\{1,\ldots,r\}$ for which:
\[\sum_{j=1}^r\vp_{i,j}^1<8\vp\mbox{ , }\sum_{j=1}^r\vp_{i,j}^2<8\vp\mbox{ , }\sum_{j=1}^r\vp_{j,i}^1<8\vp\mbox{ and }\sum_{j=1}^r\vp_{j,i}^2<8\vp.\]
From now on $i$ is fixed with this property. Noting that $y(i,j)$ is a piece of permutation, we get:
\begin{align*}
d_H(y(i,j)y(j,i)x,xy(i,j)y(j,i))&\leqslant d_H(y(i,j)y(j,i)x,y(i,j)zy(j,i))+d_H(y(i,j)zy(j,i),xy(i,j)y(j,i))\\ &\leqslant d_H(y(j,i)x,zy(j,i))+d_H(y(i,j)z,xy(i,j))=\vp_{j,i}^1+\vp_{i,j}^2.
\end{align*}
Denote by $p_j=y(i,j)y(j,i)$ and note that $y^2=Id_{nr}$ implies $\sum_jp_j=Id_n$. The above inequality is $d_H(p_j,xp_jx^*)=d_H(p_jx,xp_j)\leqslant\vp_{j,i}^1+\vp_{i,j}^2$. Analogous, $d_H(p_j,zp_jz^*)\leqslant\vp_{i,j}^1+\vp_{j,i}^2$.

For $S\subset\{1,\ldots,r\}$ define $p_S=\sum_{j\in S}p_j$. Both $p_j$ and $xp_jx^*$ are elements in $D_n$ and this implies that $d_H(p_S,xp_Sx^*)\leqslant\sum_{j\in S}d_H(p_j,xp_jx^*)$. Using the above inequalities we get that for any subset $S$: 
\[d_H(p_S,xp_Sx^*)\leqslant\sum_{j\in S}\vp_{j,i}^1+\vp_{i,j}^2<16\vp.\]
The same statement is true for $d_H(p_S,zp_Sz^*)$.
Assume that $Tr(p_S)<1/2$. Then, by hypothesis, $\lambda Tr(p_S)<d_H(p_S,xp_Sx^*)+d_H(p_S,zp_Sz^*)$. Hence $Tr(p_S)<32\vp/\lambda$ in this case. As $\sum_{j=1}^rp_j=Id_n$ it follows that there exists $j$ such that $Tr(p_j)>1-32\vp/\lambda$.

Let $w\in P_n$ be such that $d_H(w,y(i,j))<32\vp/\lambda$. It is easy to see that $d_H(wx,zw)\leqslant 32\vp/\lambda+8\vp+32\vp/\lambda<72\vp/\lambda$. The same is true for $d_H(xw,wz)$.
\end{proof}

\begin{p}\label{main theorem}
Let $\Theta:\ff_2\to\Pi_{k\to\omega}P_{n_k}$ be a sofic representation of $\ff_2$. Choose $a_k,c_k\in P_{n_k}$ such that $\Theta(a)=\Pi_{k\to\omega}a_k$ and $\Theta(c)=\Pi_{k\to\omega}c_k$. Assume that:
\begin{enumerate}
\item $\{a_k,c_k\}_k$ is an expander, i.e. there exists $\lambda>0$ such that for any $k$ for any projection $p\in D_{n_k}$ with $Tr(p)<1/2$ we have $\lambda Tr(p)<d_H(p,a_kpa_k^*)+d_H(p,c_kpc_k^*)$;
\item there is no $w\in\Pi_{k\to\omega}P_{n_k}$ such that $w\Theta(a)w^{-1}=\Theta(c)$ and $w^2\Theta(a)=\Theta(a)w^2$.
\end{enumerate}
Then there is no $\Psi$ sofic representation of $G$ such that $R([\Psi])=[\Theta]$.
\end{p}

\begin{proof}
Assume that there exists a sofic representation $\Psi:G\to\Pi_{k\to\omega}P_{n_kr_k}$ such that $\Psi|_{\ff_2}=\Theta\otimes 1_{r_k}$. Let $y=\Psi(b)$. Then $y^2=Id$ and $y\cdot[\Theta\otimes 1_{r_k}](a)=[\Theta\otimes 1_{r_k}](c)\cdot y$.

Find $y_k\in P_{n_kr_k}$ such that $y_k^2=Id_{n_kr_k}$ and $y=\Pi_{k\to\omega}y_k$.  Then $d_H(y_k(a_k\otimes 1_{r_k}),(c_k\otimes 1_{r_k})y_k)\to0$ when ${k\to\omega}$. Use Lemma \ref{main lemma} to construct $w_k\in P_{n_k}$ such that  $d_H(w_ka_k,c_kw_k)\to_{k\to\omega}0$ and $d_H(a_kw_k,w_kc_k)\to_{k\to\omega}0$. Let $w=\Pi_{k\to\omega}w_k$. Then $w\Theta(a)=\Theta(c)w$ and $\Theta(a)w=w\Theta(c)$. This is in contradiction with condition $(2)$.
\end{proof}

\subsection{Construction}

We show that there exists sofic representations of $\ff_2$ satisfying conditions $(1)$ and $(2)$ from Proposition \ref{main theorem}. Fix a sequence $\{n_k\}_k$ increasing to infinity. For each $k$, arbitrary choose two $n_k$-cycles $(a_k,c_k)$ from the $[(n_k-1)!]^2$ pairs available. It is know that the sequence $(a_k,c_k)_k$ is generating a sofic representation of the free group with probability $1$. The theory of expander graphs tells us that also the first condition required in Proposition \ref{main theorem} is attain with probability $1$ (for small enough $\lambda$).  Some estimations will show that also condition $(2)$ is satisfied with probability $1$.

\subsubsection{First condition}We review here some basic facts about expanders.

\begin{nt}
For any two cycles $a,c\in P_n$ denote by  $G_{a,c}$ the 4-regular graph $(V,E)$, where $V=\{1,\ldots, n\}$ and $E=\{(i,a(i));(i,a^{-1}(i));(i,c(i));(i,c^{-1}(i)):i\in V\}$. These graphs may have multiple edges. 
\end{nt}

\begin{de}
For a graph $G=(V,E)$ the \emph{Cheeger constant} $h(G)$ is defined as:
\[h(G)=\min_{0<|S|\leqslant\frac n2}\frac{|\partial S|}{|S|},\]
where $S\subset V$ and $\partial S$ is the set of edges in $E$ with exactly one vertex in $S$.
\end{de}

The link between Cheeger constant and condition $(1)$ from Proposition \ref{main theorem} is clear. Choose $a,c\in P_n$ and a projection $p\in D_n$. Construct $G_{a,c}=(V,E)$. Let $S$ be the subset of $V$ corresponding to $p$, so that $Tr(p)=(1/n)|S|$. We can see that $d_H(p,apa^*)+d_H(p,cpc^*)=(1/n)|\partial S|$. It follows that condition $(1)$ is satisfyed iff $\{h(G_{a_k,c_k})\}_k$ is bonded away from $0$.

The Cheeger constant is strongly connected to the spectral gap. For a graph $G$ we shall denote by $\lambda_1(G)\geqslant\lambda_2(G)\geqslant\ldots\geqslant\lambda_n(G)$ the eigenvalues of the adjacency matrix. If $G$ is a 4-regular graph, as is always the case in this paper, then $\lambda_1(G)=4$. The second eigenvalue is of interest to us.

\begin{p}(Cheeger inequality) 
For any $d$-regular graph the following holds:
\[\frac12(d-\lambda_2(G))\leqslant h(G).\]
\end{p}

The following theorem is the missing piece of the puzzle.

\begin{te}(Theorem 1.2 of \cite{Fr})
For any $\vp>0$ there exists a constant $\mu_\vp$ such that for at least $(1-\mu_\vp/n)[(n-1)!]^2$ pairs of $n$-cycles $\{a,c\}$ we have for all $i>1$:
\[|\lambda_i(G_{a,c})|\leqslant 2\sqrt3+\vp.\]
\end{te}

From now we fix $a\in P_n$ to be the cycle $(1,2,\ldots,n)$. As any two $n$-cycles are conjugate, we can deduce the following from the above theorem.

\begin{te}
There exists a constant $\mu_1$ such that for at least $(1-\mu_1/n)[(n-1)!]$ $n$-cycles $c$ we have for all $i>1$:
\[|\lambda_i(G_{a,c})|\leqslant 3.6\]
\end{te}

Altogether, setting $\lambda=0.2$, we have the following result:

\begin{p}\label{first condition}
There exists a constant $\mu_1$ such that for at least $(1-\mu_1/n)[(n-1)!]$ $n$-cycles $c$ the following holds: for any projection $p\in D_n$ with $Tr(p)<1/2$ we have $\lambda Tr(p)<d_H(p,apa^*)+d_H(p,cpc^*)$.
\end{p}

\subsubsection{Second condition}
Now we try to estimate the number of elements $w\in P_n$ so that $d_H(w^2a,aw^2)<\vp$. We stick to our choice $a=(1,\ldots, n)$.

\begin{p}\label{maximal number commuting}
Let $\vp>0$. Then the number of permutations $y\in P_n$ such that $d_H(ay,ya)<\vp$ is less than $n^{[\vp n]+1}$.
\end{p}
\begin{proof}
We construct elements almost commuting with $a$ as follows: divide $\{1,\ldots,n\}$ into $k$ subsets composed of consecutive numbers, then permute these subsets.

Formally, choose $1=i_1<i_2<\ldots<i_k<i_{k+1}=n+1$.  Define $l_j=i_{j+1}-i_j$, $j=1,\ldots k$ (the length of the $j^{th}$ segment). Define $s:\{1,\ldots,n\}\to\{1,\ldots,k\}$, $s(v)$ is the unique number such that $i_{s(v)}\leqslant v<i_{s(v)+1}$.

Choose $r\in Sym(k)$ and let $t_{r(w)}=1+\sum_{r(j)<r(w)}l_{r(j)}$ (these are the new starting points of the segments to replace the numbers $i_j$). Define:
\[y(v)=t_{r(s(v))}+(v-i_{s(v)}) .\]

If $v\in\{1,\ldots,n\}\setminus\{i_1,\ldots,i_k\}$ then $s(v-1)=s(v)$. Then $y(v-1)=t_{r(s(v))}+((v-1)-i_{s(v)})=y(v)-1$. This can be rewritten as $ya(v-1)=ay(v-1)$, so $d_H(ay,ya)\leqslant k/n$.

All permutations $y$ for which $d_H(ay,ya)\leqslant (k-1)/n$ can be constructed this way. We need to consider $k-1=[\vp n]$. The number of permutations is less than $C_n^{k-1}\cdot k!=\frac{n!\cdot k}{(n-k+1)!}<n^k$.
\end{proof}

\begin{nt}
For $y\in Sym(n)$ denote by $S_2(y)$ the number of solutions in $Sym(n)$ of the equation $x^2=y$.
\end{nt}

We now compute $S_2(y)$ for some $y\in Sym(n)$. 

\begin{nt}
For $x\in Sym(n)$ denote by $c_x(i)$ the number of $x$-cycles of length $i$.
\end{nt}

From this definition we see that $c_x(1)$ is the number of fix points of $x$ and $\sum_iic_x(i)=n$.

Let $x(1)=t$. Then $x(t)=x^2(1)=y(1)$ and $x(y(1))=x^2(t)=y(t)$. Inductively, we get:
\[x(y^k(t))=y^{k+1}(1)\mbox{ and }x(y^k(1))=y^k(t)\mbox{ for any }k\geqslant 0.\]

There are two cases. Assume that $1$ and $t$ are in the same $y$-cycle, i.e. there exists $a\in\nz$ so that $t=y^a(1)$. It follows that $x(y^{k+a}(1))=y^{k+1}(1)$ and $x(y^k(1))=y^{k+a}(1)$. Combining the two equations we get $y^{2a-1}(1)=1$. It is easy to get a contradiction if $y^k(1)=1$ for $k<2a-1$, so $1$ must be in a $y$-cycle of length $2a-1$. All the values of $x$ on the elements composing this $y$-cycle are determined ($x(1)=y^a(1)$ and the rest will follow).

Assume now that $1$ and $t$ are in two distinct $y$-cycles. Then the two cycles must be of equal length. The values of $x$ on the elements composing the two cycles are determined once we chose the value of $x(1)$.

Let's determine the number of solutions of the equation $x^2=y$ when $y$ has only cycles of length $i$. If $i$ is even then $c_y(i)$ must be even, otherwise we have no solution. We have to group these cycles in pairs of two and there are $(c_y(i))!/[(c_y(i)/2)!2^{c_y(i)/2}]$ possibilities to do so. For each coupling we have $i^{c_y(i)/2}$ associated solutions. All in all the number of solutions in this case is:
\[S_2(y)=\frac{(c_y(i))!(i/2)^{c_y(i)/2}}{(c_y(i)/2)!}.\]
If $i$ is odd then we can group $2k$ cycles in $k$ pairs for $k=0,\ldots,[c_y(i)/2]$ (here $[t]$ is the largest integer smaller than $t$). The cycles left unpaired are not adding to the number of solutions, as the permutation $x$ is perfectly determined on the elements of those cycles. We reach the formula:
\[S_2(y)=\sum_{k=0}^{[c_y(i)/2]}\frac{(c_y(i))!i^k}{(c_y(i)-2k)!k!2^k}.\]
If $y$ is an arbitrary element of $Sym(n)$ then:
\[S_2(y)=\left[\Pi_i\frac{(c_y(2i))!(i)^{c_y(2i)/2}}{(c_y(2i)/2)!}\right]\left[\Pi_i\left(\sum_{k=0}^{[c_y(2i+1)/2]}\frac{(c_y(2i+1))!(2i+1)^k}{(c_y(2i+1)-2k)!k!2^k}\right)\right]\]
iff $c_y(2i)$ is even for each $i$, otherwise $S_2(y)=0$.

\begin{p}\label{maximal number y^2=x}
The maximal number of solutions of the equation $x^2=y$ is attain when $y$ is identity. In other words: \[S_2(Id)=max\{S_2(y):y\in Sym(n)\}.\]
\end{p}
\begin{proof}
Assume first that $y$ is composed only of cycles of length $i$. So $n=ic_y(i)$. Then:
\[S_2(y)\leqslant\sum_{k=0}^{[c_y(i)/2]}\frac{(c_y(i))!i^k}{(c_y(i)-2k)!k!2^k}\]
As $S_2(Id)=\sum_{k=0}^{[n/2]}\frac{n!}{(n-2k)!k!2^k}$ it is enough to prove that:
\[\frac{(c_y(i))!i^k}{(c_y(i)-2k)!k!2^k}\leqslant\frac{n!}{(n-2k)!k!2^k}\mbox{ for }k=0,\ldots,[c_y(i)/2].\]
This inequality is equivalent to:
\[c_y(i)(c_y(i)-1)\ldots (c_y(i)-2k+1)i^k\leqslant n(n-1)\ldots (n-2k+1).\]
As $n=ic_y(i)$ we see that $n(n-i)(n-2i)\ldots (n-i(2k+1))$ is an intermediate value in the inequality above.

Let now $y$ be an arbitrary element in $Sym(n)$. By the first part of the proof $S_2(y)\leqslant\Pi_iS_2(Id_{ic_y(i)})$. The inequality $\Pi_iS_2(Id_{ic_y(i)})\leqslant S_2(Id_n)$ is clear as $\Pi_iS_2(Id_{ic_y(i)})$ counts only some of the solutions of the equation $x^2=Id_n$.
\end{proof}

\begin{nt}
Denote by $Bcyc(n,\vp)=\{c\in Sym(n):\exists w\in Sym(n), waw^{-1}=c, d_H(w^2a,aw^2)<\vp\}$.
\end{nt}

\begin{p}\label{second condition}
For small enough $\vp$ and large enough $n$ we have $Bcyc(n,\vp)<\frac1n\cdot(n-1)!$.
\end{p}
\begin{proof}
Combining Propositions \ref{maximal number commuting} and \ref{maximal number y^2=x}, we get that $Bcyc(n,\vp)< n^{[\vp n]+1}S_2(Id_n)$.

Clearly $(n-2k)!\cdot k!>[\frac n3]!$ for any $k=0,\ldots,[\frac n2]$. It follows that:
\[S_2(Id_n)=\sum_{k=0}^{[n/2]}\frac{n!}{(n-2k)!k!2^k}<\left[\frac n2\right]\cdot(n!)\cdot\left(\left[\frac n3\right]!\right)^{-1}.\]
It is easy to see that there exists a constant $t>0$ so that $[\frac n3]!>n^{tn}$ for large enough $n$ (one can use Stirlings's formula to deduce that $t$ can be chosen arbitrary close to $1/3$, but we don't need this). Altogether:
\[Bcyc(n,\vp)<n^{[\vp n]+1}\cdot n\cdot n^{-tn}\cdot (n!)=n^{[\vp n]+4-tn}\cdot\frac1n(n-1)!.\]
The conclusion can now be deduced.
\end{proof}

\begin{p}
Let $\{n_k\}_k$ be a sequence, $n_k\to\infty$ and $a_k=(1,\ldots,n_k)$. Choose $c_k\in P_{n_k}$ a random cycle from the $(n_k-1)!$ possibilities. Then $\Theta:\ff_2\to\Pi_{k\to\omega}P_{n_k}$ defined by $\Theta(a)=\Pi_{k\to\omega}a_k$ and $\Theta(c)=\Pi_{k\to\omega}c_k$ is a sofic representation satisfying the conditions in Proposition \ref{main theorem} with probability $1$.
\end{p}
\begin{proof}
Combine Propositions \ref{first condition} and \ref{second condition}.
\end{proof}

\section*{Acknowledgements}

Special thanks to Lewis Bowen and Florin R\u adulescu for important discussions and references that I used for this paper.

LIVIU P\u AUNESCU, \emph{INSTITUTE of MATHEMATICS "S. Stoilow" of the ROMANIAN ACADEMY} email: liviu.paunescu@imar.ro

\end{document}